\documentclass[runningheads]{llncs}
\usepackage{amsmath, amsfonts} 
\usepackage{blindtext}
\usepackage{enumitem}
\usepackage{subcaption}
\usepackage{graphicx}
\usepackage{soul}
\usepackage{xcolor} 
\usepackage[colorinlistoftodos,prependcaption,textsize=tiny]{todonotes} 

\begin{document}
\newtheorem{obs}[theorem]{Observation}
\newtheorem{conj}[theorem]{Conjecture}
\newtheorem{prop}[theorem]{Proposition}
\newtheorem{lem}[theorem]{Lemma}
\newtheorem{defn}[theorem]{Definition}
\newtheorem{cor}[theorem]{Corollary}

\newcommand{\disc}{\text{disc}}
\title{On the Connectivity and the Diameter of Betweenness-Uniform Graphs}
%
%
\author{David Hartman\inst{1, 2}\orcidID{0000-0003-3566-8214} \and
Aneta Pokorná\inst{1, 2}\orcidID{0000-0002-7104-8664} \and
Pavel Valtr\inst{3}\orcidID{0000-0002-3102-4166}}
\authorrunning{Hartman et al.}
%
\institute{Computer Science Institute of Charles University, Faculty of Mathematics and Physics, Charles University, Prague, Czech Republic
\url{https://iuuk.mff.cuni.cz} \and
The Institute of Computer Science of the Czech Academy of Sciences, Prague, Czech Republic \and
Department of Applied Mathematics, Charles University, Faculty of Mathematics and Physics, Charles University, Prague, Czech Republic
\url{https://kam.mff.cuni.cz}}
\maketitle              
\begin{abstract}
Betweenness centrality is a centrality measure based on the overall amount of shortest paths passing through a given vertex.
A graph is \emph{betweenness-uniform} if all its vertices have the same betweenness centrality. 
We study the properties of betweenness-uniform graphs. 
In particular, we show that every connected betweenness-uniform graph is either a cycle or a $3$-connected graph.
Also, we show that betweenness uniform graphs of high maximal degree have small diameter.


\end{abstract}
\section{Introduction and Definitions}
There are many complex networks that play a key role in our society.
Well-known examples include the Internet, systems of roads or railroads, electricity networks or social networks.
In such networks, it is often the case that information, people or goods travel
between different parts of the network, usually using shortest paths between points.
From such perspective, points with high throughput are the most important, valuable and
often also the most vulnerable parts of the network. Evaluating importance of nodes via their ability to provide information transfer might help in various application areas such as the human brain~\cite{sporns} or in construction of  utilized algorithms such as community detection algorithms~\cite{newman}.

A network can be viewed as a graph $G$ with vertex set $V(G)$ of size $n$
and edge set $E(G)$ that has maximal degree $\Delta(G)$ and
minimal degree $\delta(G)$.
Subset of vertices $S \subseteq V(G)$ is called a \emph{vertex cut},
if $G-S$ is disconnected.
Vertex connectivity of $G$, $\kappa(G)$, is minimal size of a vertex cut in $G$.
We say that $G$ is $k$-\emph{connected} if $|V(G)| > k$ and 
$G$ always remains connected after the removal of less than $k$ vertices.
For a vertex $x$, $N(x)$ stands
for the set of all vertices adjacent to $x$.
For two vertices $x, y$, the length of the shortest $x,y$-path
is their distance $d(x, y)$.
Diameter $d(G)$ of a graph $G$ is then
$\max_{x,y \in V(G)} d(x,y)$.
We denote the set $\{1, \dots, k\}$ by $[k]$.
We use ${X \choose 2}$ to denote pairs of vertices from the set $X$.

A network centrality measure is a tool helping us to assess how important
are nodes in the network. For a connected graph, the \emph{betweenness centrality} is the following centrality measure evaluating the importance of a vertex $x$ based on the amount of 
shortest paths going through it:
$$B(x) := 
\sum_{\substack{\{u,v\} \in \binom{V(G)\setminus\{x\}}{2}}}
\frac{\sigma_{u,v}(x)}{\sigma_{u,v}},$$
where $\sigma_{u,v}$ denotes the number of shortest paths between $u$ and $v$ and
$\sigma_{u,v}(x)$ denotes the number of shortest paths between $u$ and $v$ passing through $x$ \cite{freeman}.

Note that we count over each (unordered) pair $\{u,v\}$ only once.
It would be possible to count each pair both as $uv$ and as $vu$.
In such case we would obtain the betweenness value that is two times larger than in the unordered version.
Similarly, we can define betweenness centrality for an edge $e$ in a connected graph:
$$B(e) := 
\sum_{\substack{\{u,v\} \in \binom{V(G)}{2}}}
\frac{\sigma_{u,v}(e)}{\sigma_{u,v}}$$
where $\sigma_{u,v}(e)$ is the number of shortest paths between $u$ and $v$ passing through edge $e$.
Note that $B(e) \geq 1$ for every $e \in E(G)$, as the edge always forms the shortest path
between its endpoints.

There is a close relationship between betweenness centrality of edges and vertices.
By summing up the edge betweenness of all edges incident with a vertex $x$ we obtain
the adjusted betweenness centrality $B_a(x)$ of this vertex. Relation between normal and adjusted
betweenness of a vertex is
\begin{equation}\label{adjusted-normal}
B(x) =  \frac{B_a(x) - n + 1}{2}
\end{equation}
as has been shown by Caporossi, Paiva, Vukicevic and Segatto~\cite{Caporossi2012}.

Betweenness centrality is frequently used in applications, even to identify influential patients in the transmission of infection of SARS-CoV-2 \cite{covid}.
It is often studied from the algorithmic point of view \cite{apx-alg,algorithm}.
Betweenness centrality and its variants are also studied from the graph-theoretical perspective \cite{bc-variant,Barthelemy2018,balanced-workload,graph-classes,adjusted,amalgamation}.
In this paper we focus on graphs having the same betweenness on all vertices initiated in studies \cite{GagoCoronicovaHurajovaMadaras2013,CoronicovaHurajovaMadaras2018}.

A \emph{betweenness-uniform graph} is a graph, in which all vertices have the same value of betweenness centrality. Thus, betweenness-uniform graphs are graphs with all vertices being equally important in terms of the (weighted) number of shortest paths on which
they are lying. 
Networks having this property (or being close to it), are more robust and resistant to attacks, which causes betweenness-uniformity to be a promising feature for infrastructural applications.
Moreover, betweenness-uniform graphs are also interesting from theoretical point of view.
When studying the distribution of betweenness in a graph,
betweenness-uniform graphs are one of the two possible extremal cases
and have been already studied by Gago, Hurajová-Coroničová  and Madaras~\cite{GagoCoronicovaHurajovaMadaras2013,CoronicovaHurajovaMadaras2018}.
The other extremal case are graphs where each vertex has a unique value of betweenness, which
were studied by Florez, Narayan, Lopez, Wickus and Worrell~\cite{Florez2017}.

The class of betweenness-uniform graphs includes all vertex-transitive graphs,
which are graphs with the property that for every pair of vertices there exists an automorphism,
which maps one onto the other.
It is easy to see that vertex-transitive graphs are betweenness-uniform. 
Similarly, edge-transitive graphs are graphs with the property that for each pair of edges there exists an automorphism of the graph
mapping one edge onto the other. 
\begin{obs}[Pokorn\'a, 2020~\cite{Pokorna2020}]
An edge-transitive graph is betweenness-uniform if and only if it is regular.
\end{obs}
\begin{proof}
We can easily see that edge-transitive graphs are edge betweenness-uniform.
The result follows from relation~(\ref{adjusted-normal}).
\end{proof}
There are also betweenness-uniform graphs which are neither vertex- nor edge-transitive.
A construction of Gago, Hurajová-Coroničová and Madaras~\cite{GagoCoronicovaHurajovaMadaras2013} shows that, for $n$ large enough, there are superpolynomially many of these graphs of order $n$. 
Also, all distance-regular graphs are betweenness-uniform \cite{GagoCoronicovaHurajovaMadaras2013}.
Apart from the above mentioned results, not much is known about characterisation of betweenness-uniform graphs.

In this paper we prove two conjectures 
stated by Hurajová-Coroničová and Madaras~\cite{CoronicovaHurajovaMadaras2018}.
The first one is about the connectivity of betweenness-uniform graphs.
Having a connected betweenness-uniform graph, it is not too hard to show
that there cannot be any vertex cut of size one.
Consider connected components $C_0, \dots, C_p$ created by removing a cut vertex $v$.
 When we consider a vertex $a \in C_i$ for some $i \in \{0, 1, \dots, p\}$,
only pairs of vertices from $V(C_i) \setminus \{a\}$
contribute to the betweenness of this vertex. On the other hand, all pairs of vertices $\{a, b\}$ such that $a\in C_i$ and $b\in C_j$ for $i\neq j$ contribute to betweenness of the vertex $v$.
Using these two observations, along with some general bounds, we get the following property.
\begin{theorem}[Gago, Hurajová-Coroničová and Madaras, 2013~\cite{GagoCoronicovaHurajovaMadaras2013}]\label{2-conn}
Any connected betweenness-uniform graph is $2$-connected.
\end{theorem}
As we have mentioned above, vertex-transitive graphs are betweenness-uniform.
Thus, all cycles are betweenness-uniform.
In this paper we show that cycles are the only betweenness-uniform graphs which are not $3$-connected, as has been conjectured by Hurajová-Coroničová and Madaras~\cite{CoronicovaHurajovaMadaras2018}.
\begin{theorem}\label{conn}
If $G$ is a connected betweenness-uniform graph then it is a cycle or a $3$-connected graph.
\end{theorem}
A variant of this theorem has already been proven for vertex-transitive graphs and
for regular edge-transitive graphs~\cite{Pokorna2020}.
Note that there exists a betweenness-uniform graph, which is $3$-connected and 
is not vertex-transitive, see Figure~\ref{p:example}. This implies that we cannot generalize this result 
to claim that all betweenness-uniform graphs are either vertex-transitive or $4$-connected.

\begin{figure}[ht]
\centering
    \includegraphics[width=0.4\textwidth]{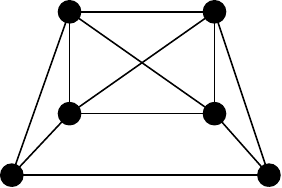}
        \caption{Example of a $3$-connected betweenness-uniform graph, which is not vertex transitive.}
  \label{p:example}
\end{figure}

The second conjecture of Hurajová-Coroničová and Madaras~\cite{CoronicovaHurajovaMadaras2018} proven in the following theorem gives a relation between the maximum degree and the diameter of a betweenness-uniform graph.
\begin{theorem}\label{t:diam}
If $G$ is betweenness-uniform graph and $\Delta(G) = n - k$, then $d(G) \leq k$.
\end{theorem}
In fact, the bound in Theorem~\ref{t:diam} can be improved to
$d(G) \leq \big \lfloor \frac{k}{3} \big \rfloor + 3 $; see Corollary~\ref{t:stronger} from Section~\ref{s:diam-deg}.

\section{Proof of Theorem~\ref{conn}}

Before we start with the proof, we introduce some definitions and
notation. 
\emph{Betweenness centrality} of a vertex $u \in V(G)$ \emph{induced by a subset of vertices} $\emptyset \neq S \subseteq V(G)$ of a graph $G$ is defined as
$$B_S(u) := \sum_{\{x,y\} \in \binom{S \setminus \{u\}}{2}} \frac{\sigma_{x,y}(u)}{\sigma_{x,y}}$$

\emph{Average betweenness} of $\emptyset \neq U \subseteq V(G)$
in $G$ is $$\bar{B}(U)\,:= \frac{\sum_{u \in U} B(u)}{|U|}.$$
\emph{Average betweenness} of $U$ \emph{induced} by $\emptyset \neq S \subseteq V(G)$
is defined analogically as 
$$\bar{B}_S(U)\,:= \frac{\sum_{u \in U}{B_S(u)}}{|U|}.$$

The main idea of the proof is to take a graph $G$ with vertex cut of 
size two and show that it is not betweenness-uniform, unless it is isomorphic to a cycle.
We denote $\{p, q\}$ to be the cut of size two in $G$ minimizing the size of the smallest connected component of $G - \{p,q\}$, which is denoted by $K$.
Let $k := |K|$ and $K^+$ be the subgraph of $G$ induced by $V(K) \cup \{p, q\}$.

\begin{obs}\label{o:smallestK}
One of the following two cases always happens:
\setlist[description]{font=\normalfont}
\begin{description}
\item [Case A:] $k = 1$
\item [Case B:] both $p, q$ have at least two neighbours in $K$.
\end{description} 
\end{obs}

\begin{proof}
Suppose that say $p$ has only one neighbour $p'$ in $K$ while $|K| > 1$. Then $p'$ together with $q$ forms a $2$-cut of $G$ such that $K'$, the smallest component of $G - \{p',q \}$, is smaller than $K$.
This is a contradiction with our choice of $p$ and $q$.
\end{proof}

\begin{figure}[ht]
\centering
\begin{subfigure}{0.3\textwidth}
  \centering
    \includegraphics[width=\textwidth]{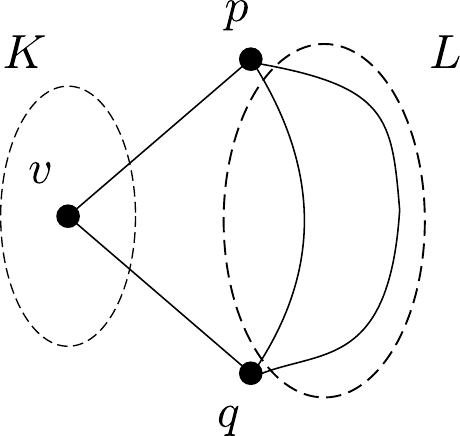}
    \caption{Case A}
\end{subfigure}
~
\begin{subfigure}{0.3\textwidth}
  \centering
    \includegraphics[width=\textwidth]{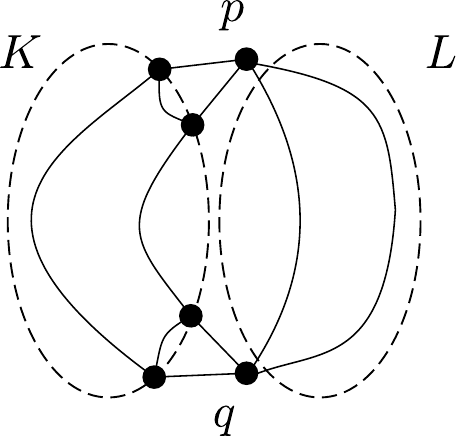}
    \caption{Case B}
\end{subfigure}
  \caption{Examples of the two possible situations from Observation~\ref{o:smallestK}.}
  \label{notation}
\end{figure}

There might exist one or more connected components in a graph $L= G \setminus \{\{p, q\}\, \cup\, V(K)\}$,
denoted by $L_1, \dots, L_j$. We denote $\ell := |L|$.
If $L$ contains more connected components $L_1, \dots, L_j$, we consider a graph
$G_i := G[V(K) \cup \{p, q\} \cup V(L_i)]$ for each of the components separately. 
The case when $L$ is not connected is discussed in Section \ref{s:L-disconnected}.
Note that each component $L_i$ of $L$ is connected, so $G_i$ is a $2$-connected graph.
We use $G:= G_i$ and $L := L_i$ for simplicity in the text below. 

See Figure~\ref{notation} for notation and the two cases of Observation~\ref{o:smallestK}.

Throughout the proof, we use a notion based on a trivial observation below.
\begin{obs}
In a betweenness-uniform graph, 
$$\bar{B}(S) = \bar{B}(R)$$
for any $\emptyset \neq R, S \subseteq V(G).$
\end{obs}
\begin{proof}
Average betweenness of any two sets of vertices is the same because the betweenness value is the same for all vertices.
\end{proof}

A discrepancy between the average betweenness of the vertices of the cut and of the vertices in component $K$ is defined as 
$$\disc := \bar{B}(\{p,q\}) - \bar{B}(V(K)).$$
Let us define
$$\disc_S := \bar{B}_S(\{p,q\}) - \bar{B}_S(V(K))$$
for $\emptyset \neq S \subseteq V(G)$.
We split the discrepancy according to which pairs of vertices contribute to it.
Namely,
\begin{equation}\label{disc}
\disc = \disc_{\binom{V(K^+)}{2}} + \disc_{\binom{V(L)}{2}} + \disc_{V(K^+) \times V(L)}
\end{equation}
where $\binom{V(K^+)}{2}$, resp. $\binom{V(L)}{2}$, denotes pairs of vertices with both vertices taken from $V(K^+)$, resp. $V(L)$,
and $V(K^+) \times V(L)$ denotes pairs with one vertex from $V(K^+)$ and 
the second from $V(L)$.

Using Observation \ref{o:smallestK}, we are going to show that if $|N(p) \cap V(K)| \geq 2$
and $|N(q) \cap V(K)| \geq 2$, then the discrepancy is always strictly positive and that 
discrepancy is zero in the case $k=1$ if and only if $G$ is isomorphic to a cycle. 

\subsection{Vertices of the Cut Have at Least Two Neighbours in $K$}\label{s:large-K}

In this section, we count the discrepancy
according to equation (\ref{disc}) and show that it is always
strictly positive if the cut-vertices $p$ and $q$ have at least two neighbours in $K$, which is Case B of Observation~{\ref{o:smallestK}}.

\subsubsection{Counting $\disc_{\binom{V(L)}{2}}$}
We take any pair of vertices $\ell_1, \ell_2$ from $V(L)$ and examine how the shortest path between them influences the discrepancy.
Basically, there are three different types of shortest paths between $\ell_1$ and $\ell_2$.
\begin{enumerate}
\item The shortest path between $\ell_1$ and $\ell_2$ passes only through vertices of $V(L)$. In this case, the shortest path does 
not influence the discrepancy, because it does not contribute to the
average betweenness of either $\{p,q\}$ or $V(K)$.
\item The shortest path between $\ell_1$ and $\ell_2$ passes through $K$, especially it enters $K$ by one cut vertex and leaves through
the second cut vertex. This path always adds one to $\bar{B}(\{p,q\})$.
If the path passes through all vertices of $K$, then it also adds
one to $\bar{B}(V(K))$. Otherwise it contributes less than one to $\bar{B}(V(K))$. Overall, such paths contribute to the discrepancy by a non-negative term.
\item  The shortest path between $\ell_1$ and $\ell_2$ passes through $p$ or $q$ without visiting component $K$.
This adds something to $\bar{B}(\{p,q\})$ and nothing to $\bar{B}(V(K))$ and thus makes a positive contribution to the discrepancy.
\end{enumerate}
Overall, we get that $\disc_{\binom{V(L)}{2}} \geq 0$. 

For the counting of $\disc_{\binom{V(K^+)}{2}}$ and $\disc_{V(K^+) \times V(L)}$ we will make use of the following observation about connectivity of $K^+$ and of two lemmas.

\begin{obs}\label{K2-conn}
Connected component $K^+$ is $2$-connected.
\end{obs}
\begin{proof}
Let us assume that there is a vertex cut $x$ of size one in $K^+$.
Observe that $x$ is also a cut of size one in $K$, because otherwise 
$x$ would separate $K$ from $\{p,q\}$ and would be also a cut of $G$,
which is a contradiction with $G$ being $2$-connected.

Let $x \in V(K)$ such that $K_1, \dots, K_J$ are connected components of $K - x$. Two situations can occur.

In the first situation, there exists a vertex of $\{p,q\}$, for example $p$,
for which there exists $K_i$, $i \in [J]$ such that $K_i \cap N(p) = \emptyset$.
Let $x = p'$ and $q = q'$. We can observe that $\{p', q'\}$
is a vertex cut of $G$ with the property that the smallest component $K' = K_i$ of $G - \{p',q'\}$ is smaller than $K$, which is a contradiction with the choice of $\{p, q\}$. This situation is shown in Figure~\ref{p:K2-conn}.

In the second situation, both $p$ and $q$ have at least one neighbour in each component of $K - x$. In this case it is clear that $K^+ - x$ is connected.

\begin{figure}[ht]
\centering
    \includegraphics[width=0.4\textwidth]{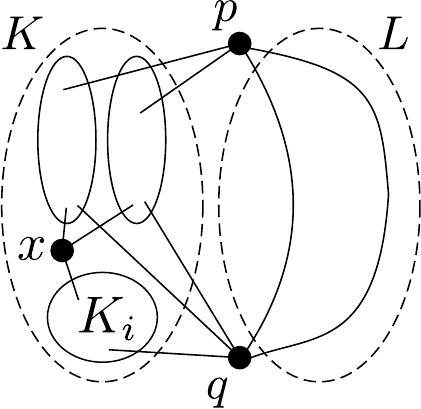}
        \caption{Illustration of the first situation in Observation~\ref{K2-conn}. Ellipsoids inside $K$ represent the components of $K - \{x\}$. Clearly, $x$ and $q$ separate $K_i$ from the rest of $G$ and $|K_i| < K$.}
  \label{p:K2-conn}
\end{figure}
\end{proof}

The observation above allows us to use the following lemma giving a bound on average distance to a vertex in a $2$-connected graph.
\begin{lem}\label{avg-dist-2-conn}
Let $G$ be a $2$-connected graph on $n$ vertices and $u \in V(G)$.
\renewcommand{\theenumi}{\roman{enumi}}
\renewcommand{\labelenumi}{\theenumi)}
\begin{enumerate}
\item for $n$ even,
        $$\frac{\sum_{v \in V(G)} d(v,u)}{n} \leq \frac{n}{4}$$
\item for $n$ odd,
        $$\frac{\sum_{v \in V(G)} d(v,u)}{n} \leq \frac{n}{4} - \frac{1}{4n}$$
\end{enumerate}
and equality is obtained for $G$ isomorphic to a cycle.
\end{lem}
Note that this result, although in a slightly less precise formulation,
was also independently given by Plesník in 1984. \cite{Plesnik}

\begin{proof}
We take the longest cycle $C$ in $G$ containing $u$.
Let $d = |C|$. We can observe that the (non-decreasing)
sequence of distances from $u$ to the other vertices of $C$ has one of the following forms:
\renewcommand{\theenumi}{\roman{enumi}}
\renewcommand{\labelenumi}{\theenumi)}
\begin{enumerate}
    \item For $d$ odd, the sequence is
    $$1, 1, 2, 2, \dots, \frac{d-1}{2} ,  \frac{d-1}{2}
    = \Big \{\Big \lfloor \frac{2t+1}{2} \Big \rfloor |\ t \in \big \{1, \dots, \frac{d}{2} - 1 \big\} \Big \}
    $$
    \item For $d$ even, the sequence is
    $$1, 1, 2, 2, \dots, \frac{d}{2}-1, \frac{d}{2} - 1, \frac{d}{2}
    = \Big \{\Big \lfloor \frac{2t+1}{2} \Big \rfloor |\ t \in \big \{1, \dots, \frac{d}{2}\big \} \Big \}
    $$
\end{enumerate}
From the $2$-connectivity, each vertex $v$ in $V(G) \setminus V(C)$ lies on a cycle $C'$ containing $u$.
Furthermore, $|C'| \leq |C|$ from the choice of $C$, so $d(u, v) \leq \lfloor\frac{d}{2}\rfloor$.
We get that the sum of distances to $u$ in $G$ is upper bounded by the sum $$1 + 1 + 2 + 2 + \dots + \frac{d-1}{2} + \frac{d-1}{2} + \sum_{i=d+1}^{n-1} \Big\lfloor\frac{d}{2}\Big\rfloor$$
for $d$ odd and by the sum 
$$1 + 1 + 2 + 2 + \dots + \Big(\frac{d}{2}-1\Big) + \Big(\frac{d}{2} - 1\Big) + \frac{d}{2} + \sum_{i=d+1}^{n-1} \Big\lfloor\frac{d}{2}\Big\rfloor$$
for $d$ even,
because the distances between $u$ and vertices on $C$ are upper bounded by the sequences mentioned above and distance between $u$ and any other vertex in $G$ is at most $\lfloor\frac{d}{2}\rfloor$.
The sum is maximized when $d = n$, in which case it is $\frac{n^2}{4}$ for $n$ even, and $\frac{n^2}{4}-\frac14$ for $n$ odd. The first part of the lemma follows. Since the above sequences contain the distances in a cycle, the second part of the lemma also holds.
\end{proof}
Note that by multiplying this result by the number of vertices $n$
we obtain a bound on the sum of distances to a fixed vertex in $G$.
By multiplying the result by $n^2$ we obtain a bound on the total sum of distances in $G$ where we count $d(u,v)$ and $d(v,u)$ as two distinct values.
The latter bound can be further used in the following lemma, which gives 
a direct relation between average betweenness and average distance in 
a graph.

\begin{lem}[Comellas, Gago, 2007 \cite{avg-dist-bc}]\label{avg-bc}
For a graph $G$ of order $n$,
$$\bar{B}(V(G)) = \frac{(n-1)}{2} \cdot \Bigg(\frac{\sum_{(u,v) \in V(G)^2} d(u,v)}{n(n-1)} - 1\Bigg)$$
\end{lem}

Now the ground is set for the calculation of the remaining two parts of discrepancy.

\subsubsection{Counting $\disc_{\binom{V(K^+)}{2}}$}
Using Observation~\ref{K2-conn} and Lemma \ref{avg-dist-2-conn} we obtain that 
the sum of the lengths of all shortest paths from a fixed vertex in $K^+$ is $\frac{(k+2)^2}{4}$.
Moreover, by multiplying the sum of shortest paths from a fixed vertex by the number of vertices in $K^+$, we obtain that the sum of all shortest paths in $K^+$ is at most $\frac{(k+2)^3}{4}$.

To obtain an upper bound on $\bar{B}_{\binom{V(K^+)}{2}}(V(K))$, 
we utilize the relation from Lemma~\ref{avg-bc}, using the sum
of distances in $K^+$, but $k$ as the number of vertices.
This corresponds to dividing all the contributions of shortest paths in $K^+$
only to vertices of $K$. Note that
some of the shortest paths might pass though $p$ or $q$, but this can 
only decrease the average betweenness of $K$.
As a result, 
$$\bar{B}_{\binom{V(K^+)}{2}}(V(K)) \leq \frac{k^2}{8} + \frac{k}{4} + \frac{1}{k} + 2.$$
Finally, we assume that $\bar{B}_{\binom{V(K^+)}{2}}(\{p,q\}) = 0$ to obtain a lower bound on the discrepancy.
This results in
$$\disc_{{\binom{V(K^+)}{2}}} \geq 0 - 
\Big(\frac{k^2}{8} + \frac{k}{4} + \frac{1}{k} + 2 \Big) = -\frac{k^2}{8} - \frac{k}{4} - \frac{1}{k} - 2.$$

\subsubsection{Counting $\disc_{V(K^+) \times V(L)}$}
Clearly, each path from $K$ to $L$ passes through at least one vertex of the cut $\{p, q\}$,
adding at least $\frac{1}{2}$ to $\bar{B}_{V(K^+) \times V(L)}(\{p,q\})$.
As a result, $\bar{B}_{V(K^+) \times V(L)}(\{p,q\}) \geq \frac{k \cdot \ell}{2}$.

Now we show an upper bound on the average betweenness of $K$.
\begin{obs}\label{o:dist}
Let $x \in V(L)$ and suppose $d(x,p) < d(x,q)$.
The contribution of $x$ to $\bar{B}(V(K))$ is maximized, when
all paths from $K^+$ to $x$ pass through $p$.
\end{obs}
\begin{proof}
Otherwise, there exists $y \in V(K)$ such that the shortest path between $x$ and $y$ passes through $q$.
This means that $d(y,q) + d(x,q) \leq d(y,p) + d(x,p)$. Together with $d(x,q) \geq d(x,p)$, we get 
$d(y, q) \leq d(y, p)$, so the path passes through smaller or the same number of vertices of $K$, than it would if it went through $p$.
See Figure~\ref{p:dist} for illustration.
\end{proof}

\begin{figure}[ht]
\centering
    \includegraphics[width=0.4\textwidth]{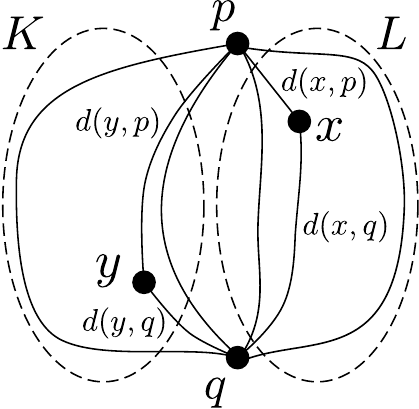}
        \caption{Illustration of the situation from Observation~\ref{o:dist}}
  \label{p:dist}
\end{figure}

From now on, we assume that for each $x \in V(L)$ there exists $r \in \{p,q\}$ such that all paths from $K^+$ to $x$ are passing through $r$ 
and we denote $s := \{p,q\} \setminus r$.

We can use Lemma~\ref{avg-dist-2-conn} and the fact that $K^+$ is $2$-connected to obtain that
$\sum_{v \in V(K^+)} d(v,r) = \sum_{v \in V(K) \cup \{s\}} d(v,r) \leq \frac{(k+2)^2}{4}.$
This corresponds to a sum of distances travelled inside $K$ by all paths from $V(K)$ to a fixed $x \in V(L)$.
Note that for any $v \in V(K)$, the sum of contributions of all $v,r$-paths of length $d$ to the betweenness of $K$  is $d - 1$ and thus
$$\bar{B}_{V(K^+) \times V(L)}(V(K)) \leq 
 \frac{\ell}{k} \Big(\sum_{v \in V(K) \cup \{s\}} (d(v,r) - 1) \Big)
 \leq \frac{\ell}{k} \Big(\frac{(k+2)^2}{4} - (k+1) \Big)
= \frac{k\ell}{4}$$

When we take the lower bound on $\bar{B}_{V(K^+) \times V(L)}(\{p,q\})$ and upper bound on $\bar{B}_{V(K^+) \times V(L)}(V(K))$ we can bound the discrepancy 
$$\text{disc}_{V(K^+) \times V(L)} \geq \frac{k \ell}{2} - \Big(\frac{k \ell}{4} \Big) = \frac{k \ell}{4}.$$

Together we obtain 
$$\text{disc} = \text{disc}_{\binom{V(K^+)}{2}} + \text{disc}_{\binom{V(L)}{2}} + \text{disc}_{V(K^+) \times V(L)}
\geq -\frac{k^2}{8} - \frac{k}{4} - \frac{1}{k} - 2
+ 0 + \frac{k \ell}{4} > 0$$
which holds for $k=2, \ell \geq 8$; $k \geq 3, \ell \geq 6$ and $l \geq k \geq 5$.
It remains to discuss the cases $k = 2, \ell \leq 7$ and $k \in \{3,4,5\}, \ell \leq 5$.

For $k = 2, \ell \leq 7$ and $k \in \{3,4\}, \ell \leq 5$ we have used the database provided by Brendan McKay~\cite{small-graphs} to filter all betweenness-uniform graphs up to ten vertices and verify that the theorem holds for all such graphs.
For the case $k \in \{4, 5\}, l=5$ we have used the program nauty~\cite{nauty} to generate all $2$-connected graphs on $k$ vertices as choices for $K$ and all connected graphs on seven vertices as choices of $L \cup \{p, q\}$ and for all choices of $p$, $q$ and their possible adjacencies in $K$ verified the theorem on graphs of this form.

We have seen that discrepancy is always greater than zero if $|N(p) \cap K| \geq 2$ and $|N(q)\  \cap K| \geq 2$, implying $G$ is not betweenness-uniform in this case. 
From this fact and Observation \ref{o:smallestK}, we see that the size of the minimal connected component $K$ must be one if the $2$-connected graph $G$ is betweenness-uniform.

\subsection{Component $K$ Consists of a Single Vertex}
Here we consider the Case A from Observation~\ref{o:smallestK} giving $|K|=1$ which means that $G$ contains a vertex $v \in K$ of degree $2$ for which $N(v) = \{p,q\}$. Let $K^+:=G[\{p,v,q\}]$ and let $L:=G[V\setminus\{p,q,v\}]$.

Let us recall that discrepancy is defined as $\bar{B}(\{p,q\}) - \bar{B}(V(K))$, which in the case considered in this section equals to $(B(p)+B(q))/2 - B(v)$. Also note that in this case $\disc_{\{p\}\times\{q\}} \leq -1$, where equality is obtained when $pvq$ is the single shortest path between $p$ and $q$. Similarly to the previous section, we will decompose discrepancy to two parts, one part induced by the pairs of vertices from $L$ and second part induced by pairs where at least one vertex is part of $K^+$:
$$\disc = \disc_{\binom{V(L)}{2}} + \disc_{V(K^+) \times V(L)\, \cup\, \binom{V(K^+)}{2}}$$
Observe that $\disc_{\binom{V(L)}{2}} \geq 0$
because every (shortest) path between two vertices  
of $L$ going through the vertex $v$ of $K$ must contain both $p$ and $q$.
This enables us to focus only on paths with at least one end-vertex in the set $K^+ = \{p,q,v\}$.

We start with the following observation, which enables us to rule out the case $pq \in E(G)$.

\begin{obs}\label{Kn-only-BU-zero}
A betweenness-uniform graph $G$ of order $n$ has the values of betweenness centrality equal to zero if and only if $G$ is isomorphic to $K_n$.
\end{obs}
\begin{proof}
It has been already shown that a vertex has betweenness value zero if and only if it's neighborhood forms a clique.
\cite{EverettSeidman1985,Grassietal2009}
It is clear that neighborhood of each vertex in $K_n$
forms a clique and thus $K_n$ is betweenness-uniform with
betweenness value zero.

Suppose we have a betweenness-uniform graph $G$ with betweenness
value zero where exist $u,v \in V(G)$ such that
$uv \notin E(G)$. If there are more missing edges,
take $u, v$ whose distance is minimal.
If $d(u,v) > 2$, then exists $v'$ on a shortest path
between $u$ and $v$ such that $uv' \notin E(G)$ and
$d(u, v') < d(u,v)$. As a result, $d(u,v) = 2$ and
thus exists $w$ such that $u,v \in N(w)$.
However, $uv \notin E(G)$, so $B(w) > 0$, a contradiction.
\end{proof}

Observe that if $G$ contains the edge $pq$, then there is no shortest path passing through $v$ and thus $B(v)=0$. As a consequence of Observation~\ref{Kn-only-BU-zero}, $G$ is isomorphic to $K_3$.
Therefore, we assume that $G$ does not contain the edge $pq$ from now on. Let $P$ be the shortest path from $p$ to $q$ in $G-v$.

Now, the main goal is to show that there exists a vertex (almost) in the middle of the path $P$, which is a vertex-cut of $G[V(L)]$.
\begin{lem}\label{L-cut}
There exists  $t \in V(L)$ such that $|d(t,p) - d(t,q)| \leq 1$ and $\{v,t\}$ is a vertex-cut of $G$.
\end{lem}

\begin{proof}
For a vertex $w$ in $G$, set $\alpha(w):=d(w,p)-d(w,q)$.
Let $\lambda$ be the number of vertices of $P$ different from $p$ and $q$. Along the path $P$ from $p$ to $q$, the function $\alpha$ consecutively gets values $-\lambda-1,-\lambda+1,\dots,\lambda-1,\lambda+1$.
It follows that, depending on the parity of $\lambda$, one of the following two cases happens: 
\setlist[description]{font=\normalfont}
\begin{description}
\item [Case 1:]  There  are  two consecutive vertices $x,y$ on $P$ such  that $$\alpha(x)  = -1,\ \alpha(y)  = 1.$$
\item [Case 2:] There  are  three  consecutive vertices $x,y,z$ on $P$ such  that $$\alpha(x)  =-2,\ \alpha(y)  = 0,\ \alpha(z) = 2.$$
\end{description}

Before considering Cases 1 and 2 separately, we prove the following proposition.

\begin{prop}\label{p:degree2-delta}
Let $w\in V(L)$.
\setlist[description]{font=\normalfont}
\begin{description}
\item [(i)] If $|\alpha(w)|\leq 1$, then
$$\disc_{V(K^+)\times \{w\}}=\frac12.$$

\item [(ii)]If $\alpha(w)=-2$, then
$$\disc_{V(K^+)\times \{w\}}=\frac12\left(1-\frac{\sigma_{w,p}}{\sigma_{w,q}}\right)\in \left(0,\frac12\right).$$

\item [(iii)]
If $\alpha(w)=2$, then
$$\disc_{V(K^+)\times \{w\}}=\frac12\left(1-\frac{\sigma_{w,q}}{\sigma_{w,p}}\right)\in \left(0,\frac12\right).$$

\item [(iv)]
If $|\alpha(w)|\geq 3$, then
$$\disc_{V(K^+)\times \{w\}}=0.$$
\end{description}
\end{prop}

\begin{proof}
\setlist[description]{font=\normalfont,style=unboxed,leftmargin=0cm}
\begin{description}
\item [(i)] If $|\alpha(w)|\leq 1$, then all shortest paths from $w$ to $p$ avoid $q$ as well as $v$ and
all shortest paths from $w$ to $q$ avoid $p$ and $v$.
Hence, shortest paths from $w$ to $\{p, q\}$ do not contribute to the betweenness of vertices $v, p$ and $q$.
Every shortest path from $w$ to $v$ goes through exactly one of the vertices $p$ and $q$. 
Part (i) of the proposition follows.

\item [(ii)] If $\alpha(w)=-2$, then $w$ is closer to $p$ than to $q$ and thus every shortest path from $w$ to $p$ avoids $q$ and $v$. For the same reason, every shortest path from $w$ to $v$ visits $p$ and avoids $q$, so the total contribution of these 
paths is $1/2$.

There are two types of shortest paths between $w$ and $q$.
First type passes through both $p$ and $v$, second type avoids both of them.
As $d(w,p)+2=d(w,q)$, 
the number of shortest paths of the first type
is the same as $\sigma_{w,p}$, the total number of shortest paths from $w$ to $p$.
As a result, exactly $\sigma_{w,p}$
of the $\sigma_{w,q}$ shortest paths from $w$ to $q$ visit $p$ and $v$. Each of these $\frac{\sigma_{w,p}}{\sigma_{w,q}}$ paths
contributes one to $B(v)$ and $B(p)$, which means it contributes one half to $\bar{B}(\{p,q\}) =  (B(p) + B(q))/2$.
As a result, the contribution of each of these paths to the discrepancy is $-1/2$.
Note that $\sigma_{w,p} < \sigma_{w,q}$, because each shortest path from $w$ to $p$ can be extended by $v$ and $q$ to a shortest 
path between $w$ and $q$. Also, there must be at least one shortest $w,q$-path avoiding $p$ and $v$.

All the other shortest paths from $w$ to $q$ avoid both $p$ and $v$, so they do not influence the discrepancy. Part (ii) follows.
\item [(iii)] Analogous to the proof of part (ii), with the roles of $p$ and $q$ exchanged.

\item [(iv)] If $\alpha(w)\leq -3$ then $w$ is much closer to $p$ than to $q$, so all shortest paths from $w$ to $p$ visit none of the vertices $v$ and $q$, which does not influence discrepancy. Also, all shortest paths from $w$ to $v$ visit $p$ and do not visit $q$, contributing $1/2$ to $(B(p)+B(q))/2$,
and all shortest paths from $w$ to $q$ visit both $p$ and $v$, contributing $1$ to $B(v)$ and $1/2$ to $(B(p)+B(q))/2$.
Note that each shortest $w,v$ can be uniquely extended to a shortest $w,q$-path, so the number of these paths is the same.
Part (iv) then follows.

The case $\alpha(w)\geq 3$ is analogous, with the roles of $p$ and $q$ exchanged. 
\end{description}
\end{proof}

Now we continue with the proof of Lemma \ref{L-cut}.
Suppose that Case 1 holds.
According to Proposition~\ref{p:degree2-delta} (i),
$$\disc_{V(K^+)\times \{x,y\}}=\frac{1}{2}+\frac{1}{2}=1.$$
Further, we have $\disc_{\{p\}\times \{q\}}=-1$, since the only shortest path between $p$ and $q$ goes through $v$.
Proposition~\ref{p:degree2-delta} now implies that if $\disc=0$ then every vertex $w$ in $V(L)\setminus\{x,y\}$ satisfies $|\alpha(w)|\geq 3$.
Since the function $\alpha$ differs by at most two on any pair
of neighbors in $L$,
$\{v,x\}$ and $\{v,y\}$ are $2$-cuts of $G$.
This finishes the proof of Lemma \ref{L-cut} for Case 1.

Suppose now that Case 2 happens. 
We have $\disc_{\{p\}\times \{q\}}\geq-1$,
because the $\disc_{\{p\}\times \{q\}} = -\frac{1}{2}$ for $p = x, q=z$ and otherwise it is $-1$.
According to Proposition~\ref{p:degree2-delta} (i),
$\disc_{V(K^+)\times \{y\}}=1/2$. 
According to Proposition~\ref{p:degree2-delta} (ii),(iii),
$$\disc_{V(K^+)\times \{x,z\}}=
\frac12\left(1-\frac{\sigma_{x,p}}{\sigma_{x,q}}\right)+\frac12\left(1-\frac{\sigma_{z,q}}{\sigma_{z,p}}\right),$$
which is positive.
It now follows from Proposition~\ref{p:degree2-delta} (i) that $|\alpha(w)|\geq2$ for every vertex $w\in V(L), w\neq y$, since otherwise $\disc>0$. 

We have $\sigma_{x,q}\ge \sigma_{x,p}+\sigma_{z,q}$, 
because each shortest $x,p$-path can be extended by $v$ and $q$ to a shortest $x,q$-path and each shortest $z,q$-path can be extended by $y$ and $x$ to a shortest $x,q$-path.
This holds because $d(x,p) = d(z,q)$.
The equality is obtained if and only if all shortest paths from $x$ to $q$ go either through $z$ or $p$.
Similarly, $\sigma_{z,p}\ge \sigma_{x,p}+\sigma_{z,q}$.
Thus, 
$$\disc_{V(K^+)\times \{x,z\}}=
\frac12\left(1-\frac{\sigma_{x,p}}{\sigma_{x,q}}\right)+\frac12\left(1-\frac{\sigma_{z,q}}{\sigma_{z,p}}\right) \geq$$
$$\geq \frac12\left(1-\frac{\sigma_{x,p}}{\sigma_{x,p}+\sigma_{z,q}}\right)+\frac12\left(1-\frac{\sigma_{z,q}}{\sigma_{x,p}+\sigma_{z,q}}\right)=\frac12.$$

Proposition~\ref{p:degree2-delta} now implies that if $\disc=0$ then every vertex $w$ in $V(L)\setminus\{x,y,z\}$ satisfies $|\alpha(w)|\geq 3$. It follows that 
$\{v,x\},\{v,y\},\{v,z\}$ are $2$-cuts in $G$.

This finishes the proof of Lemma \ref{L-cut} for Case 2.
\end{proof}

Now we use Lemma \ref{L-cut} to show that if $|K|=1$ in $G$ with a $2$-cut $\{p,q\}$ such that $K$ is the smallest component of $G - \{p,q\}$, $G$ is either isomorphic to a cycle, or not betweenness-uniform. The latter will be done by showing $\text{disc} > 0$.

If $G$ is isomorphic to a cycle, then it fulfills the condition of Theorem \ref{conn} and we are done.
Suppose $G$ is not isomorphic to a cycle and let us take the vertex $t$ from Lemma \ref{L-cut}.
The graph $G-\{v,t\}$ consists of two connected components $L_p, L_q$
such that $p \in L_p, q \in L_q$. If $\alpha(t) = -1$, we consider $t \in L_p$ and if $\alpha(t) = 1$, we consider $t \in L_q$.
If $\alpha(t) = 0$, $t \notin L_p, L_q$. 

Since $G$ is not a cycle, at least one of the following two conditions is satisfied:
\setlist[description]{font=\normalfont}
\begin{description}[align=left]
\item [(C1)] $L_p$ is not a path, or
\item [(C2)] $L_q$ is not a path.
\end{description}

If both (C1) and (C2) hold, then without loss of generality, we suppose that $|V(L_p)|\leq |V(L_q)|$. Then we consider the 
graph $H:=G[V(L_p)\cup\{v,t\}]$. 
If, let us say, (C1) holds and (C2) does not hold, then there is a single shortest path from $t$ to $q$. It follows that $|V(L_p)|\leq|V(L_q)|$, since otherwise $d(t,q) > d(t,p)$, $pvq\dots t$ forms a path, $L_p$ is not a path and thus $B(p)>B(q)$.
We now can again consider the graph $H:=G[V(L_p)\cup\{v,t\}]$.

If $H$ is $2$-connected, we can continue in the same way as in Case B of the Observation~\ref{o:smallestK}, which always yields positive discrepancy as has been shown in previous subsection.
Note that this is possible as we have only used the fact that $K$ is the smallest component of $G-\{p,q\}$ to deduce that it is $2$-connected.

If $H$ is not $2$-connected,
there is a cut vertex $t'$ of $H$. 
Recall that in order to be betweenness-uniform, $G$ must be $2$-connected and thus it also holds that $t'$ separates $p$ and $t$.
Let $H_p$ and $H_t$ be the two connected components of $H-t'$, where $p\in H_p$ and $t\in H_t$. Since $H$ is not a path from $p$ to $t$, at least one of the two subgraphs of $G$ induced by $V(H_p)\cup\{t'\}$ and by $V(H_t)\cup\{t'\}$, respectively,
is not a path from $p$ to $t'$ and from $t$ to $t'$, respectively.
We consider such a subgraph and check again if it is $2$-connected. Continuing this process, due to the $2$-connectivity of $G$, we end up with a $2$-cut of $G$ which cuts off a $2$-connected component, which is a subgraph of $H$. Then we can continue as 
in Case B of Observation~\ref{o:smallestK}
which leads to $\text{disc} > 0$ and thus $G$ is not betweenness-uniform. This finishes the proof of Theorem \ref{conn}
for Case A of Observation~\ref{o:smallestK}.

\subsection{There are More Connected Components in $L$}\label{s:L-disconnected}
If $G - (\{p,q\} \cup V(K))$ had only one connected component $L$,
as we have assumed in the previous parts of this section,
we are finished.
Otherwise, we know that $L = \bigcup_{i=1}^j L_i$ and 
any $G_i := G[V(K) \cup \{p, q\} \cup V(L_i)]$ has either positive discrepancy, or it is isomorphic to a cycle.
We can observe that for $G_i$ and $G_j$ with positive discrepancy,
$$G_{i+j} := G[V(K) \cup \{p,q\} \cup V(L_i) \cup V(L_j)]$$
has also positive discrepancy. 
This follows from the fact that whenever we obtain positive discrepancy, it is due to the Case B of Observation~\ref{o:smallestK}, which has been solved in Subsection~\ref{s:large-K}.
From there, it is clear that discrepancy rises 
with growing difference between $k$ and $\ell$.
The only case when discrepancy is zero for each $G_i$ is when
 each $L_i$ is isomorphic to a path between $p$ and $q$.
Recall that $\disc_{V(L) \choose 2} \geq 0$, because any path between vertices of $L$ passing through $K$ also passes through both $p$ and $q$.
Suppose $\ell_i \in N(p) \cap V(L_i)$ and $\ell_j \in N(p) \cap V(L_j)$
for any two connected components $L_i, L_j$ of $L$. 
Then the shortest path between $\ell_i$ and $\ell_j$ passes through $p$ and 
avoids $K$, making $\disc_{\binom{V(L)}{2}} > 0$, which leads to $\disc > 0$.

By the results above, any $2$-connected graph has either $\disc > 0$, implying it is not betweenness-uniform, or it has $\disc = 0$ and is isomorphic to a cycle.
This finishes the proof of Theorem \ref{conn}.

\section{Relation between Maximal Degree and Diameter of Betweenness-Uniform Graphs}\label{s:diam-deg}
In this section we prove a conjecture of Hurajová-Coroničová and Madaras~\cite{CoronicovaHurajovaMadaras2018}
claiming that betweenness-uniform graphs with high maximal degree have small diameter.
\begin{conj}[\cite{CoronicovaHurajovaMadaras2018}]\label{diam-deg}
If $G$ is a betweenness-uniform graph and $\Delta(G) = n - k$, then $d(G) \leq k$.
\end{conj}
In a previous article by Gago, Hurajová-Coroničová and Madaras \cite{GagoCoronicovaHurajovaMadaras2013},
this conjecture has been proved for $k = 1$ and $k = 2$
and later Hurajová-Coroničová and
Madaras \cite{CoronicovaHurajovaMadaras2018} proved
the conjecture for $k = 3$ by showing that
a betweenness-uniform graph
with $\Delta(G) = n - 3$ has $d(G) = 2$ for $n \geq 4$.

Before proving Conjecture~\ref{diam-deg}, we state the following more general result.
\begin{theorem}\label{t:gen-conn}
Let $G$ be a $\ell$-connected graph with $\Delta(G) = n-k$. Then $d(G) \leq \lfloor \frac{k-3}{\ell} \rfloor + 4$.
\end{theorem}

\begin{proof}
For $k = 2$ we have $d(G) \leq 3$ and for $k = 1$ we have $d(G) \leq 2$, so the claim holds for these values of $k$. In the rest of the proof we assume $k \geq 3$.

Let $y$ be a vertex such that $\deg(y) = n-k \geq 3$ and let $u,v$ be a pair of vertices such that $d(u, v) = d(G)$.
Thus, every path between $u$ and $v$ has at least $d(G)+1$ vertices. We now show that at least $d(G)-4$ of these vertices do not lie in the set $S:=\{u,v,y\}\cup N(y)$.
Let $P$ be a path between $u$ and $v$. If $P$ contains at most five vertices of $S$ then
it contains at least $(d(G)+1)-5=d(G)-4$ vertices outside of $S$. Suppose now that $P$ contains more than five vertices of $S$. Let $Q$ be the shortest subpath of $P$ covering all the vertices of $P$ lying in $S\setminus\{u,v\}=\{y\}\cup N(y)$. Let $a$ and $b$ denote the end-vertices of the path $Q$. They both lie in $\{y\}\cup N(y)$. If $y$ is one of the two vertices $a$ and $b$ then we denote by $Q'$ the path consisting of a single edge $ab$. Otherwise we denote by $Q'$ the path $ayb$. In each case, $Q'$ is a path in $G$ containing only vertices of $\{y\}\cup N(y)$. Now, let $P'$ be the path between $u$ and $v$ in $G$ obtained from $P$ by replacing the subpath $Q$ by $Q'$. The path $P'$ contains at most five vertices of $S$. Therefore, it contains at least $(d(G)+1)-5=d(G)-4$ vertices not lying in $S$. All these vertices lie also in $P$. This concludes the proof that every path between $u$ and $v$ has at least $d(G)-4$ vertices not lying in $S$.

As $G$ is $\ell$-connected, there are at least $\ell$ vertex disjoint paths between $u$ and $v$.
They together contain at least $\ell(d(G) -4)$ vertices not lying in $S$. 
Since the size of $S$ is $(n-k+3)$, we obtain a lover bound on the number of vertices, $$n \geq \ell(d(G) -4)  + (n-k+3)$$
giving $d(G) \leq \frac{k - 3}{\ell} + 4$. Since $d(G)$ is an integer, we get $ d(G)
\leq \big\lfloor \frac{k - 3}{\ell} \big\rfloor + 4$.
\end{proof}

\begin{cor}\label{c:3-conn}
Let $G$ be a $3$-connected graph with $\Delta(G) = n-k$. 
Then $d(G) \leq \lfloor \frac{k}{3} \rfloor + 3$.
\end{cor}

Using Corollary~\ref{c:3-conn} and Theorem~\ref{conn}, we can obtain a result, which is stronger than Conjecture~\ref{diam-deg} for $k \geq 4$.
For $k \leq 3$, the conjecture has already been proved, as we have mentioned above, which makes the statement true for all values of $k$.
\begin{cor}\label{t:stronger}
Let $G$ be a betweenness-uniform graph with $\Delta(G) = n - k\geq 3$.
Then $d(G) \leq \big\lfloor \frac{k}{3} \big\rfloor + 3$.
\end{cor}
Note that the assumption $\Delta(G) \geq 3$ is necessary, 
because $C_n$ is betweenness-uniform,
has $\Delta(G)= n - k = 2$, implying $k=n-2$, and its diameter is at most 
$\lfloor \frac{n}{2} \rfloor = \lfloor \frac{k}{2} \rfloor + 1$, which is greater than
$\big\lfloor \frac{k}{3} \big\rfloor + 3$ for $k > 12$.
\begin{proof}
Theorem~\ref{conn} and the assumption $\Delta(G)\geq 3$ imply that $G$ is $3$-connected.
The rest follows from Corollary~\ref{c:3-conn}.
\end{proof}

\begin{prop}
The bound of Theorem~\ref{t:gen-conn} is tight for $(k - 3)\,|\,\ell$.
\end{prop}

\begin{proof}
Given $k$ and $n$ we create a graph $G$ containing vertices $u$ and $v$ in distance 
$d = \lfloor \frac{k-3}{\ell} \rfloor + 4$ that contains $\ell$ vertex disjoint $u,v$-paths $P_1, \dots, P_\ell$ of length $d$. We take vertices $y_1, \dots, y_\ell$ such that $d(u, y_i) = \lfloor \frac{d}{2} \rfloor$ and $y_i \in P_i$ for $i \in [\ell]$, which are the midpoints of the paths. Next we define $x_i \in P_i$, respectively $z_i \in P_i$, as vertices satisfying $d(u, x_i) = \lfloor \frac{d}{2} \rfloor - 1$, respectively $d(u, x_i) = \lfloor \frac{d}{2} \rfloor + 1$.

Next, we add edges between $y_i$ and vertices $x_j, y_{i'}, z_{j'}$  for $i,j, i', j' \in [\ell]$ and $i \neq j' $, resulting in $\deg(y_i) = 3\ell - 1$ for all $i \in [\ell]$.
Finally we add set of new vertices $\{w_j\ |\ j \in [n-k-3\ell+1] \}$ and 
edges $y_{i}w_{j}$ for $i \in [\ell], j \in [n-k-3\ell+1]$.

From the construction we get that $\Delta(G) = n-k$ and $d(G) = \lfloor \frac{k-3}{\ell} \rfloor + 4$.
\end{proof}





\begin{cor}
Assume that the fact that $G$ is betweenness-uniform can only be utilized
in bounding $d(G)$ to obtain the fact that $G$ is $3$-connected.
Then the bound of Corollary~\ref{t:stronger} is tight.
\end{cor}

\subsubsection{Acknowledgments.} David Hartman and Aneta Pokorná were partially supported by the Czech Science Foundation Grant No. 23-07074S. DH and AP were partially supported by ERC Synergy grant DYNASNET grant agreement no. 810115. This article is extended version of the proceeding paper \cite{caldam2021}.

%
%
%

 \bibliographystyle{splncs04}
 \typeout{}
 \bibliography{betweenness-uniform-properties}
 
\end{document}